\documentclass[12pt,a4paper]{article}
\usepackage[top=2.5cm,bottom=2.5cm,left=2.5cm,right=2.5cm]{geometry}

\usepackage{amssymb,amsmath,graphicx,color,amsthm,centernot}
\usepackage[capitalise]{cleveref}
\usepackage[labelformat=simple]{subcaption}
\usepackage[font=small,labelfont=bf,width=12.5cm]{caption}

\usepackage{multirow,tabularx}

\newtheorem{claim}{Claim}
\newtheorem{theorem}{Theorem}[section]
\newtheorem{lemma}[theorem]{Lemma}
\newtheorem{corollary}[theorem]{Corollary}
\newtheorem{proposition}[theorem]{Proposition}

\newtheorem{problem}[theorem]{Problem}
\newtheorem{question}[theorem]{Question}

\newcommand{\set}[1]{\ensuremath{\left\{#1 \right\}}}

    {\noindent \emph{Proof.} {}{#1}{}}{$~$\hfill $~\blacklozenge$ \vspace{0.2cm}}
    
\newcommand{\qedclaim}{$~$\hfill $~\blacklozenge$ \vspace{0.2cm}}

\definecolor{defblue}{rgb}{0.4,0,0.84}
\definecolor{greyblue}{rgb}{0.23,0.4,0.70}
\definecolor{orange}{rgb}{1.0,0.5,0.2}
\definecolor{violet}{rgb}{0.55,0,0.55}

\usepackage{url}
\makeatletter
\g@addto@macro{\UrlBreaks}{\UrlOrds}
\makeatother

\newcolumntype{Y}{>{\centering\arraybackslash}X}

\begin{document}

\title{{\bf Note on robust coloring of planar graphs}}

\author
{
	Franti\v{s}ek Kardo\v{s}\thanks{CNRS, LaBRI, University of Bordeaux, Talence, France. E-mail: \!\texttt{frantisek.kardos@u-bordeaux.fr}} \thanks{Comenius University, Faculty of mathematics, physics and informatics, Bratislava, Slovakia.}\quad
	Borut Lu\v{z}ar\thanks{Faculty of Information Studies in Novo mesto, Slovenia. E-mail: \texttt{borut.luzar@gmail.com}} \thanks{Rudolfovo Institute, Novo mesto, Slovenia.}\quad
	Roman Sot\'{a}k\thanks{Pavol Jozef \v Saf\'{a}rik University, Faculty of Science, Ko\v{s}ice, Slovakia. E-mail: \texttt{roman.sotak@upjs.sk}}
}

\maketitle

{
\begin{abstract}
	We consider the robust chromatic number $\chi_1(G)$ of planar graphs $G$ 
	and show that there exists an infinite family of planar graphs $G$
	with $\chi_1(G) = 3$, thus solving a recent problem of Bacs\'{o}~et~al. 
	(The robust chromatic number of graphs, Graphs Combin. 40 (2024), \#89).
\end{abstract}
}

\medskip
{\noindent\small \textbf{Keywords:} robust coloring, robust chromatic number, Tutte graph, planar graph.}

\section{Introduction}

In this paper, we consider simple graphs; i.e., graphs without loops and multiedges.
A mapping $f \, : \, V(G) \rightarrow E(G) \cup \set{\emptyset}$ is a {\em $1$-selection} of a graph $G$
if for every $v \in V(G)$ either $f(v)$ is incident with $v$ or $f(v) = \emptyset$.
In other words, $f$ is an assigment of at most one incident edge for every vertex of a graph.
We call the resulting set of edges $f(V(G))$ a {\em $1$-selection set}.
Clearly, each component of the graph induced by any $1$-selection set is a graph with at most one cycle.

Given a $1$-selection $f$, the graph $G_f$ obtained from $G$ by deleting the edges from $f(V(G))$ 
is called the {\em $1$-removed subgraph} of $G$ regarding $f$.
Using the set of all $1$-removed subgraphs of $G$, 
we define the {\em robust chromatic number} of a graph $G$ as
$$
	\chi_1(G) = \min_f \chi(G_f)\,,
$$
where $\chi(G_f)$ is the chromatic number of $G_f$.

The notion of robust coloring was recently introduced by Patk\'{o}s, Tuza, and Vizer~\cite{PatTuzViz24},
who used it as a tool for deriving estimates on a Turan-type extremal problem.
Afterwards, a systematic investigation of the robust chromatic number and other robust invariants 
was done by Bacs\'{o} et al. in~\cite{BacPatTuzViz23},
and in~\cite{BacBujPatTuzViz23}, results for some specific graph classes were presented.

In particular, for the robust chromatic number, the following two general results were derived.
\begin{theorem}[Bacs\'{o} et al.~\cite{BacPatTuzViz23}]
	For any graph $G$, it holds that
	$$
		\bigg\lceil \frac{\chi(G)}{3} \bigg\rceil \le \chi_1(G) \le \chi(G)\,.
	$$
	Moreover, both bounds are tight.
\end{theorem}

Recall that a graph is {\em $d$-degenerate} if its every subgraph contains a vertex of degree at most $d$.
\begin{theorem}[Bacs\'{o} et al.~\cite{BacPatTuzViz23}]
	\label{thm:deg}
	For any $d$-degenerate graph $G$, it holds that
	$$
		\chi_1(G) \le \frac{d}{2} + 1\,.		
	$$
	Moreover, this upper bound is tight as for every integer $k \ge 1$ 
	there exists a $2k$-degenerate graph $H_k$ with $\chi_1(H_k) = k+1$.
\end{theorem}

As a direct corollary of Theorem~\ref{thm:deg}, one obtains upper bounds 
on outerplanar, triangle-free planar, and planar graphs,
since, by Euler's formula, they are $2$-degenerate, $3$-degenerate, and $5$-degenerate, respectively.
\begin{corollary}~
	\label{cor:plan}
	\begin{itemize}
		\item[$(i)$] For any outerplanar graph $G$, it holds that $\chi_1(G) \le 2$.~\cite{BacPatTuzViz23}
		\item[$(ii)$] For any triangle-free planar graph $G$, it holds that $\chi_1(G) \le 2$.
		\item[$(iii)$] For any planar graph $G$, it holds that $\chi_1(G) \le 3$.~\cite{BacPatTuzViz23}
	\end{itemize}
\end{corollary}

The first two cases of Corollary~\ref{cor:plan} are both tight,
whereas for the case $(iii)$,
it was not clear whether the upper bound can be achieved by some planar graph;
Bacs\'{o}~et~al.~\cite{BacPatTuzViz23} proposed this question as a problem.
\begin{problem}[Bacs\'{o} et al.~\cite{BacPatTuzViz23}]
	Do there exist planar graphs with $\chi_1(G) = 3$, or is $2$ a universal upper bound?
\end{problem}

The truth of the latter part would provide an interesting feature that any planar graph
admits a $1$-selection set whose removal results in a bipartite planar graph.
This however is not the case as we show in this note.
\begin{theorem}
	\label{thm:main}
	There is an infinite family of planar graphs $G$ with $\chi_1(G) = 3$.
\end{theorem}

Let us mention here that, as observed by Voigt~\cite{Voi24},
the above result can also be obtained through a different approach, 
using a result of Kemnitz and Voigt~\cite{KemVoi18} on list coloring of planar graphs.

The rest of the paper is organized as follows. 
In Section~\ref{sec:prel}, we define notions used further on,
in Section~\ref{sec:proof}, we prove Theorem~\ref{thm:main},
and in Section~\ref{sec:conc}, we conclude with some directions of further work.

\section{Preliminaries}
\label{sec:prel}

In this section, we introduce the notions and terminology that we use in our proofs.

For a set of edges $S$, by $G - S$ we denote the graph obtained by removing the edges of $S$ from $G$,
and by $G[S]$ we denote the subgraph of $G$ induced by the edges of $S$.

A {\em matching} in a graph $G$ is any $1$-regular subgraph of $G$,
and a {\em $k$-factor} of $G$ is a $k$-regular spanning subgraph of $G$.

For a {\em plane graph} $G$, i.e., a planar graph together with some embedding in the plane, 
with $V(G)$, $E(G)$, and $F(G)$ we denote its set of vertices, edges, and faces, respectively.
The edges bounding a face $f$ of a plane graph are the {\em boundary edges} of $f$.
In a connected plane graph, the length of a shortest closed trail 
along the edges bounding a face $f$ is the {\em length of $f$}, denoted by $\ell(f)$.
If a graph is not connected, then the length of a face $f$ is the sum of the lengths of shortest closed trails bounding $f$.
A face of length $k$ is called a {\em $k$-face}.
In a $2$-connected plane graph, the boundary edges of every face form a cycle~\cite{MohTho01},
while in non-connected plane graphs face boundaries may be comprised of several (disconnected) parts.
In particular, if a boundary of a face contains a bridge, that edge is counted twice in the length of a face
as the boundary trail passes it twice.
On the other hand, as shown by Whitney~\cite{Whi33},
a $3$-connected planar graph has a unique (up to equivalence) embedding in the plane. 

If every face in a plane graph $G$ is of length $3$, then $G$ is a {\em triangulation},
and similarly, $G$ is a {\em quadrangulation} if its every face is of length $4$.
Recall that every plane quadrangulation is bipartite.

A {\em (geometric) dual} $G^*$ of a plane graph $G$ is a plane (multi-)graph with $V(G^*) = F(G)$, $F(G^*) = V(G)$ 
and two vertices $f^*,g^*$ in $G^*$ being connected by an edge for every edge incident with the both corresponding faces $f,g$ in $G$
(so, each edge $e \in E(G)$ has a corresponding edge $e^* \in E(G^*)$).
For a set $S$ of edges from $G$, we denote the set of the corresponding edges in $G^*$ by $S^*$.
Moreover, if $G$ is a triangulation,
then $G^*$ is a bridgeless cubic graph
and therefore it contains a perfect matching~\cite{Pet1891}.
Note also that the dual of $G^*$ is the graph $G$.

\section{Proof of Theorem~\ref{thm:main}}
\label{sec:proof}

Before we give the proof of the theorem, 
we state several auxiliary results.
We first formalize the observation from the definition of $1$-selections.
\begin{proposition}
	\label{prop:unicyc}
	Let $H$ be a subgraph of a graph $G$ induced by some subset of edges of $G$.
	Then, every component of $H$ contains at most one cycle if and only if 
	$H$ is a $1$-selection set.
\end{proposition}

\begin{proof}
	Assume first that every component of $H$ contains at most one cycle.	
	If $C$ does not contain a cycle, then we select an arbitrary vertex of the tree $C$ to be $v_r$, 
	representing the root of the considered tree.
	If $C$ contains a cycle, then we direct its edges in a circular way, 
	and then we temporarily collapse the cycle into a single vertex $v_r$, 
	the root of the tree obtained this way. 
	Now we direct all the edges of $C$ away from $v_r$.
	
	Next, we define a $1$-selection $f$ such that to every vertex $v$ of $C$
	we assign the edge whose terminal vertex is $v$.
	Note that in the case when $C$ does not contain a cycle, the vertex $v_r$ has no edge assigned.
	This establishes the right implication.
	
	For the left implication, assume that $H$ is a $1$-selection set and
	suppose the contrary that there is a component $C$ of $H$ containing at least two cycles.
	In this case, $|E(C)| \ge |V(C)| + 1$, meaning that at least two edges need 
	to be assigned to the same vertex, a contradiction.
\end{proof}

We say that a $1$-selection $f$ guarantees a property $\mathcal{P}$ for a graph $G$
if $G_f$ has the property $\mathcal{P}$;
analogously, a $1$-selection set guarantees a property $\mathcal{P}$ for $G$.
In the following two lemmas, we use the notion of a {\em minimal $1$-selection set},
by which we mean a $1$-selection set $S$ which guarantees bipartiteness of $G-S$ for a graph $G$
but removal of any edge from $S$ would not guarantee bipartiteness anymore.

\begin{lemma}
	\label{lem:K2OrClaw}
	Let $G$ be a plane triangulation. 
	If $S \subseteq E(G)$ is a minimal $1$-selection set which guarantees bipartiteness for $G$ (i.e., $G - S$ is bipartite),
	then every $3$-face in $G$ has exactly one or three of its edges in $S$.
	Moreover, there are at most two $3$-faces with all three edges in $S$.
\end{lemma}

\begin{proof}
	Let $n = |V(G)|$.
	Since $G$ is a plane triangulation, it follows from Euler's formula that the number of $3$-faces in $G$ is exactly $2n-4$.
	If $G-S$ is bipartite, then every $3$-face must have at least one edge in $S$.
	Since every edge is incident with two faces, this means that we need at least $n-2$ edges to cover every face
	(note that covering all faces with exactly $n-2$ edges means that the dual of $G$ admits a perfect matching).
	So, $n-2 \le |S| \le n$.
	
	Suppose now that $uvw$ is a $3$-face with exactly two edges, say $uv$ and $vw$, in $S$.
	Then, in any $2$-coloring of the (bipartite) graph $G - S$, the vertices $u$ and $w$
	receive distinct colors.
	Therefore, $v$ has distinct color as $u$ or $w$, say $u$, and thus the edge $uv$ needs not be in $S$.
	This means that $S - uv$ also guarantees bipartiteness for $G$, which contradicts the minimality of $S$.
	
	Finally, observe that if all three edges of a $3$-face are in $S$, they cover four $3$-faces altogether,
	and thus we can have at most two $3$-faces with all three edges in $S$, 
	since $n-2 \le |S| \le n$.
	Moreover, the two $3$-faces are not adjacent by Proposition~\ref{prop:unicyc}.
\end{proof}

From Lemma~\ref{lem:K2OrClaw} it follows that the graph $G^*[S^*]$ induced by the edges of $S^*$
is comprised only of isolated edges and at most two claws $K_{1,3}$.

For the purposes of stating and proving the next lemma, 
we refer to isolated vertices as {\em degenerate cycles}.

\begin{lemma}
	\label{lem:3faces}
	Let $G$ be a plane triangulation. 
	If $S \subseteq E(G)$ is a minimal $1$-selection set which guarantees bipartiteness for $G$ (i.e., $G - S$ is bipartite),
	then every face of the graph $G^* - S^*$ has a boundary consisting of at most two (possibly degenerate) cycles.
\end{lemma}

Note that, by Lemma~\ref{lem:K2OrClaw}, a degenerate cycle may appear on the boundary of some face
at most twice (in the cases of $G^*[S^*]$ having components isomorphic to $K_{1,3}$).

\begin{proof}
	Recall that since $G$ is a triangulation, $G^*$ is cubic. 
	Moreover, since $S^*$ covers all the vertices of $G^*$ and $G[S^*]$ has only vertices of degrees $1$ and $3$,
	the graph $G^* - S^*$ has maximum degree $2$ and the boundary of every face is comprised of
	cycles.
	
	Now, suppose to the contrary that $g$ is a face of $G^* - S^*$ whose boundary consists of at least three (possibly degenerate) cycles.	
	Let $\mathcal{C}_g^* = \set{C_1^*,C_2^*,\dots,C_\ell^*}$ be the set of the boundary cycles of $g$.
	Next, let $R^* \subseteq S^*$ be the set of all edges of $S^*$ incident with the set of faces of $G^*$ corresponding to the face $g$ of $G^* - S^*$.
	Clearly, the set $R^*$ is nonempty.
	Observe that the set $R$ in $G$ corresponding to the set $R^*$ 
	induces a connected subgraph $G[R]$ of the graph $G[S]$,
	which has at most one cycle by Proposition~\ref{prop:unicyc}.
	
	Now, let $R^*_i \subset R^*$ be the edges with exactly one endvertex incident with the cycle $C_i^* \in \mathcal{C}_g^*$.
	Observe that the edges of every $R^*_i$ form cuts in $G^*$
	and thus the corresponding sets $R_i$ form subgraphs of $G$ with at least one cycle in each of their components and hence also in $G[R]$.	
	Therefore, we may assume that every $R_i$ forms a connected component with exactly one cycle, otherwise we already have at least two cycles in $G[R]$, a contradiction.
	
	It remains to show that $G[R]$ contains at least two cycles, since it may happen that 
	for some pair of sets $R_j^*$ and $R_k^*$, for some $j,k\in\set{1,\dots,\ell}$, 
	we have that $R_j^* \subseteq R_k^*$ or $R_k^* \subseteq R_j^*$, 
	and thus the sets $R_j$ and $R_k$ contribute only one cycle.
	However, note that since every edge in $R_j^* \cap R_k^*$ has one endvertex in $R_j^*$ and the other in $R_k^*$,
	it cannot appear in any other $R_i^*$, and therefore, there is at least one additional cycle in $G[R]$,
	which by Proposition~\ref{prop:unicyc} means that $S$ is not a $1$-selection, a contradiction.
\end{proof}

Now, we are ready to give a proof of Theorem~\ref{thm:main}.
\begin{proof}[Proof of Theorem~\ref{thm:main}]
	In order to prove the theorem, we construct a cubic $3$-connected plane graph $G^*$ whose dual $G$
	has $\chi_1(G) = 3$.
	In the construction, we use copies of $T'$ (see the right configuration in Figure~\ref{fig:Tutte}).
	The configuration $T'$ has the property that if a graph $G$ contains $T'$ as a subgraph, 
	then every path in $G$, which uses two of the edges $a$, $b$, $c$ and visits all vertices of $T'$
	contains the edge $a$~\cite{Tut46,Ros90}. In other words, every Hamiltonian cycle in $T$ (see the left graph in Figure~\ref{fig:Tutte})
	contains the edge $a$.
	\begin{figure}[htp!]
		$$
			\includegraphics{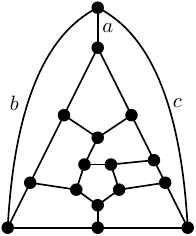} \quad\quad\quad
			\includegraphics{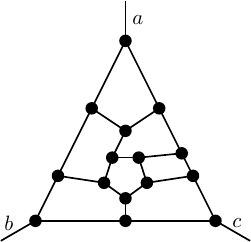}
		$$
		\caption{The Tutte's triangle $T$ (left) and the configuration $T'$ (right).}
		\label{fig:Tutte}	
	\end{figure}

	
	In Figure~\ref{fig:Tuttes}, all possible matchings in $T'$ containing the edge $a$ are depicted,
	and this implies the following claim.
	\begin{claim}
		\label{cl:2fac}
		Every $2$-factor in $T'$, not containing the edge $a$, contains at least one cycle 
		not using the edges $b$ and $c$.
		\qedclaim
	\end{claim}
	\begin{figure}[htp!]
		$$
			\includegraphics{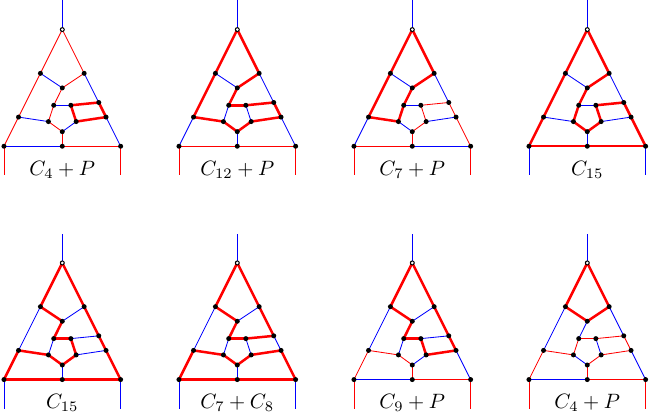}
		$$
		\caption{Possible matchings (blue edges) in $T'$ that contain the edge $a$, and the corresponding $2$-factors (red edges)
			with internal cycles depicted heavier.}
		\label{fig:Tuttes}	
	\end{figure}	
	
	Now, consider the graph $G^*$ in Figure~\ref{fig:theGraph}.
	\begin{figure}[htp!]
		$$
			\includegraphics{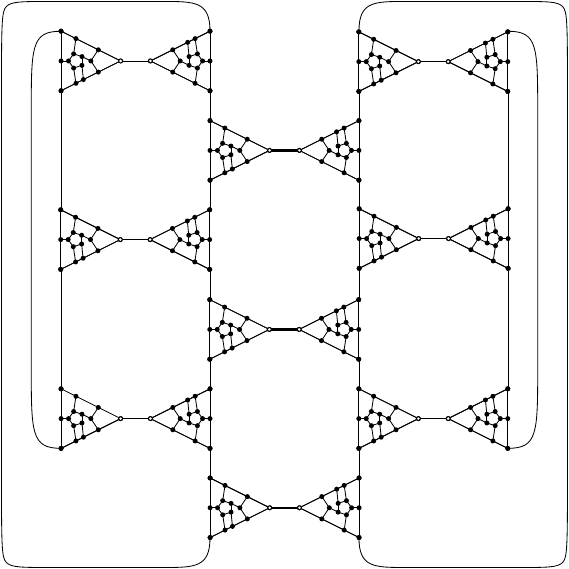}
		$$
		\caption{The dual $G^*$ of a plane triangulation $G$ with $\chi_1(G) = 3$.
			The three edges depicted heavier form a $3$-edge-cut in which at least one of the edges is also in $R^*$.}
		\label{fig:theGraph}	
	\end{figure}		
	It contains three $3$-edge-cuts each comprised of three edges $a$ and copies of $T'$ on both sides of each edge $a$.
	Since at most two vertices in $G^*[S^*]$ have degree $3$,
	this means that at least six copies of $T'$ corresponding to one $3$-edge-cut (say the middle one, depicted with heavier edges) 
	have the property that for any $1$-selection set $R$ guaranteeing bipartiteness for $G$
	at least one edge $a$, 
	denote it $a_x$, in that $3$-edge-cut is in the corresponding set $R^*$.
	
	Therefore, by Claim~\ref{cl:2fac}, the two corresponding copies of $T'$ 
	contain a cycle in the boundary of some face of $G^* - R^*$
	and consequently, the face incident with $a_x$ has at least three boundary cycles in $G^* - R^*$.
	By Lemma~\ref{lem:3faces}, this means that $G - R$ is not bipartite.
	
	In order to obtain an infinite family of plane graphs with the robust chromatic number equal to $3$,
	observe that $G^*$ can easily be further expanded by, e.g., copies of $3$-edge-cuts 
	and so the corresponding duals will still not be bipartite.
\end{proof}

On the other hand, there is a big class of plane graphs having
the robust chromatic number equal to $2$.
\begin{proposition}
	\label{prop:2}
	For every plane triangulation $G$ whose dual admits a $2$-factor with at most two cycles, 
	we have that $\chi_1(G) = 2$.
\end{proposition}

\begin{proof}
	Let $C^*$ be a $2$-factor in $G^*$ with at most two cycles. 
	Then $S^* = E(G^*) \setminus E(C^*)$ is a perfect matching in $G^*$.
	Observe that there is at most one cycle in $S$ corresponding to the cut in $S^*$ represented by the edges with endvertices incident with
	two cycles of $C^*$.
	Therefore, $S$ is a $1$-selection set by Proposition~\ref{prop:unicyc}.
	Moreover, since $S^*$ is a perfect matching, every $3$-face in $G$ has one incident edge in $S$
	and so the graph $G-S$ is quadrangulation, hence bipartite.
%
\end{proof}

Clearly, all subgraphs of triangulations from Proposition~\ref{prop:2} also have the robust chromatic number at most $2$.
\begin{corollary}
	Every planar graph $G$, which is a subgraph a plane triangulation whose dual admits a $2$-factor with at most two cycles, 
	has $\chi_1(G) \le 2$.
\end{corollary}

\section{Conclusion}
\label{sec:conc}

The investigation of robust invariants is in an early stage,
but a number of intriguing results have already been obtained.
However, there are plenty of further open problems in this area.

For example, based on the result presented in this note,
one may wonder what is the upper bound for the robust chromatic number of
graphs embeddable to surfaces of higher genus.
In the case of the torus, every graph embeddable on it is $6$-degenerate,
and with some additional analysis, using the result of Thomassen~\cite{Tho94} about proper coloring of toroidal graphs,
one can derive that every toroidal graph $G$ has $\chi_1(G) \le 3$.

On the other hand, for $K_{10}$, which has genus $4$ and non-orientable genus $7$ (see, e.g.,~\cite{MohTho01}), 
we have that $\chi_1(K_{10}) = 4$ by~\cite[Theorem~2]{BacPatTuzViz23}. 
So, the following question arises.
\begin{question}~
\begin{itemize}
	\item[$(a)$] Is there a graph $G$ with genus less than $4$ and $\chi_1(G) = 4$?
	\item[$(b)$] Is there a graph $G$ with non-orientable genus less than $7$ and $\chi_1(G) = 4$?
\end{itemize}
\end{question}

In particular, the result of Theorem~\ref{thm:deg} for degenerate graphs does not seem to be tight 
for degenerate graphs of given genus.
So we propose also the following.
\begin{problem}
	Determine the tight upper bound for the robust chromatic number of 
	graphs with given (non-orientable) genus.
\end{problem}

\paragraph{Acknowledgement.} 
The authors would like to thank the referee for a careful reading of the manuscript.
F. Kardo\v{s} was supported by the Slovak VEGA Grant 1/0743/21 and by the Slovak Research and Development Agency under the contract No. APVV--19--0308.
B.~Lu\v{z}ar was partially supported by the Slovenian Research Agency Program P1--0383 and the projects J1--3002 and J1--4008.
R.~Sot\'{a}k was supported by the Slovak VEGA Grant 1/0574/21 and by the Slovak Research and Development Agency under the contracts No. APVV--19--0153 and APVV--23--0191.
	
\bibliographystyle{plain}
{
	\bibliography{References}
}

\end{document}